\documentclass{article}

\usepackage[utf8]{inputenc}
\usepackage[T1]{fontenc}
\usepackage{amssymb}
\usepackage[english]{babel}
\usepackage{amsmath}
\usepackage{amsthm}
\usepackage{epsfig}
\usepackage{mathrsfs}
\usepackage{xcolor}
\usepackage{graphicx}
\usepackage{enumerate}
\usepackage{cases}
\usepackage{dsfont}
\usepackage[toc,page]{appendix}
\usepackage{hyperref}

\makeatletter
\newtheorem*{rep@theorem}{\rep@title}
\newcommand{\newreptheorem}[2]{%
\newenvironment{rep#1}[1]{%
 \def\rep@title{#2 \ref{##1}}%
 \begin{rep@theorem}}%
 {\end{rep@theorem}}}
\makeatother

\def\RR{\mathbb{R}}

\def\div{\mathrm {div}\,}
\def\supp{\mathrm {supp}\,}

\def\L{\mathcal L}

\def\pa{\partial}
\def\na{\nabla}
\def\eps{\varepsilon}
\def\vphi{\varphi}
\def\bO{\overline{\Omega}}
\def\dO{\pa\Omega}

\def\Dels{\big(-\Delta\big)^s}

\def\Leps{\mathcal{L}^\eps}

\def\le{\lambda^\eps}

\def\wpsi{\widetilde \psi}

\def\bF{F_1}

\def\d{\,{\rm{d}}}

\theoremstyle{plain}
\newtheorem{theorem}{Theorem}[section]
\newtheorem{lemma}{Lemma}[section]
\newreptheorem{lemma}{Lemma}
\newtheorem{proposition}{Proposition}[section]
\newreptheorem{proposition}{Proposition}
\newtheorem{definition}{Definition}[section]

\title{Fractional Diffusion limit of a kinetic equation with Diffusive boundary conditions in a bounded interval}
\date{}
\author{
L. Cesbron\thanks{Department of Mathematics, ETH Z\"urich, Switzerland}, 
A. Mellet\thanks{Department of Mathematics and CSCAMM, University of Maryland,  USA. Partially supported by NSF Grant DMS-2009236.}, 
M. Puel\thanks{Département de Mathématiques, CY Cergy Paris Université, France.}
} 

\makeindex
\begin{document}
\maketitle

\begin{abstract}
We investigate the fractional diffusion approximation of a kinetic equation set in a bounded interval with diffusive reflection conditions at the boundary.
In an appropriate singular limit corresponding to small Knudsen number and long time asymptotic, we show that the asymptotic density function is  the {\it unique solution} of a fractional diffusion equation with Neumann boundary condition.
This analysis completes a previous work by the same authors in which a limiting  fractional diffusion equation was identified on the half-space, 
but the uniqueness of the solution (which is necessary to prove the convergence of the whole sequence) could not be established.
\end{abstract}

\section{Introduction}
\subsection{The linear Boltzmann equation with diffusive boundary conditions}
In this paper, we investigate the fractional diffusion approximation of a linear kinetic equation set on a bounded domain with diffusive boundary conditions in dimension $1$.
Our starting point is the following kinetic equation, which models the evolution of a particle distribution function $f(t,x,v)\geq 0$ depending on the time $t>0$, the position $x\in\Omega\subset \RR$ and the velocity $v\in\RR$:
\begin{equation}\label{eq:kin0}
\begin{cases}
\displaystyle  \pa_t f   + v\, \pa_x f   = \nu_0 \left( \int_\RR f (t,x,w)\, dw \, F(v) - f \right)  
& \mbox{ in } \RR_+\times\Omega\times\RR \\[8pt]
 f  (0,x,v) = f_{in} (x,v) &\mbox{ in }\Omega\times\RR .
\end{cases}
\end{equation}
The left hand side of \eqref{eq:kin0} models the free transport of particles, whereas the operator in the right hand side models the diffusive and mass preserving interactions between the particles and the background.
For simplicity, we consider here the linear Boltzmann operator with constant collision frequency $\nu_0>0$ and equilibrium function $F(v)$.
Importantly, the function $F(v)$ is taken to be a given heavy-tail distribution function satisfying, for some $s\in(1/2,1)$ and $\gamma>0$:
\begin{equation}\label{def:F}
\left\{ \begin{aligned}
& F \in L^{\infty}(\RR), \quad \int F(v) \d v = 1, \quad F(v) = F(|v|)\geq 0 \\
& \Big| F(v) - \frac{\gamma}{|v|^{1+2s}} \Big| \leq \frac{C}{|v|^{1+4s}} \qquad \mbox{ for all } |v|\geq 1.
\end{aligned} \right. 
\end{equation}

Importantly, we consider here the case where  $\Omega$ is a bounded interval and we take (without loss of generality) $\Omega = (0,1)$. We denote $\Gamma_{\pm} = \{(x,v)\in\dO \, ;\, \pm n(x)\cdot  v>0\}$ (note that $\dO = \{ 0, 1\}$ and $n(0)=-1$, $n(1)=1$)
and define the traces $\gamma_\pm f = f|_{\Gamma_\pm}$.
With these notations, we consider the following diffusive reflection conditions on $\pa \Omega$:
\begin{equation} \label{eq:BCdiff}
\gamma_-f(t,x,v) = \mathcal B [\gamma_+ f] (t,x,v) \qquad \forall (x,v)\in \Sigma_-
\end{equation}
where $\mathcal B $ is the following scattering operator
\begin{align} \label{def:diffBC}
\mathcal{B}[\gamma_+ f^\eps ](t,x,v) = c_0 F(v) \int_{w\cdot n(x)>0} \gamma_+ f^\eps  (t,x,w) |w\cdot n(x)| \d w  
\end{align}
with $c_0$ the normalizing constant: 
\begin{align} \label{def:c0} 
c_0 :=  \left( \int_{w\cdot n(x)>0} F(w) |w\cdot n(x)| \d w \right)^{-1} .
\end{align}
The use of diffusive reflection conditions at the boundary is classical in kinetic theory. We are assuming that the 
boundary operator $\mathcal{B}$ involves the same equilibrium function $F$ as the bulk collision operator in order to avoid the need of boundary layer analysis. Note that we consider $s>1/2$ in order for the constant $c_0$ to be well-defined. 

\medskip

The diffusion approximation of such an equation is obtained by investigating the long time, small mean-free-path asymptotic behavior of $f$. To this end we introduce the Knudsen number $\eps\ll1$ and the following rescaling of \eqref{eq:kin0}-\eqref{eq:BCdiff}:
\begin{equation} \label{eq:rescaledKinetic} 
\left\{ \begin{aligned} 
& \eps^{2s}  \pa_t f^\eps  + \eps v\pa_x f^\eps  = \nu_0 \left( \int_\RR f^\eps(t,x,w)\, dw F(v) - f^\eps \right)  &\mbox{ in } \RR_+\times\Omega\times\RR \\
& f^\eps  (0,x,v) = f_{in} (x,v) &\mbox{ in }\Omega\times\RR \\[5pt]
& \gamma_- f^\eps  (t,x,v) = \mathcal{B}[\gamma_+ f^\eps ](t,x,v) &\mbox{ on } \RR_+\times
\Gamma_- 
\end{aligned} \right. 
\end{equation} 
We see that the particular choice of power of $\eps$ in front of the time derivative in \eqref{eq:rescaledKinetic} depends on the equilibrium $F$. When $\Omega$ is the whole line $\RR$ it has been proved (see for instance \cite{MelletMischlerMouhot11,Mellet10,AbdallahMelletPuel11,BouinMouhot20} and references therein) that as $\eps$ goes to $0$, $f^\eps$ converges to a function $f^0(t,x,v)= \rho(t,x)F(v)$
where $\rho(t,x)$ is the weak solution of a fractional diffusion equation $\pa_t \rho + \kappa \Dels \rho = 0$.

There is now a very significant literature devoted to the fractional diffusion approximation of kinetic equations. But the role of boundary conditions in these limits has only recently started to be investigated.
The case of Dirichlet boundary condition was studied in \cite{AcevesSchmeiser17}
and the case of specular reflection conditions was investigated by the first author in \cite{Cesbron18,Cesbron20}.
In \cite{CMP1}, we considered the case of diffusive reflection conditions \eqref{eq:BCdiff}
 in dimension $n \geq 1$ when $\Omega$ is the half space $\{x_n>0\}$. 
However, while this previous work clearly identified the limiting Neumann fractional diffusion equation  in $\Omega$ (see Section \ref{sec:As} below), we did not prove that the limiting density was the unique weak solution of that equation (given, for instance,  by Hille-Yoshida's theorem). We only established that it satisfies the equation in a weaker sense, for which uniqueness is not clear. As a result, we also did not prove the  convergence of the whole sequence $f^\eps$.
\medskip

The goal of this paper is to fill this gap in the simpler one-dimensional framework by proving that the limiting density is the unique weak solution of a Neumann fractional diffusion equation.
We achieve this by sharpening the assumptions on the test functions used to derive the limiting equation.
In addition, this paper provides  the first result of this type in a bounded domain.
Finally, we point out that while we focus here on the one-dimensional case, the proofs provide a roadmap for handling this problem in higher dimensions and in general convex domains.

\subsection{Weak solutions of \eqref{eq:rescaledKinetic}}
We now recall the standard definition of weak solutions for the kinetic equation with diffuse boundary condition.
First, we note that for any test function $\phi \in \mathcal{D}(\RR_+\times\overline \Omega \times\RR)$, smooth solutions of \eqref{eq:rescaledKinetic} satisfy:
\begin{align*}
&- \iiint_{ \RR_+\times  \Omega\times\RR} f^\eps \pa_t \phi \d t \d x \d v - \iint_{\Omega\times\RR}  f_{in}(x,v) \phi(0,x,v) \d x \d v   \\
&\qquad\qquad + \eps^{1-2s} \iint_{\RR_+\times\Gamma_+} \gamma_+ f^\eps  \left(\gamma_+ \phi - \mathcal{B}^*[\gamma_-\phi] \right) |v\cdot n(x)| \d t \d \sigma(x) \d v \\
&\qquad\qquad\qquad = \eps^{-2s} \iiint_{ \RR_+\times  \Omega\times\RR}  \big[ f^\eps \left( \eps v\pa_x \phi  -  \nu_0 \phi \right) + \nu_0 \rho^\eps F(v) \phi \big] \d t \d x \d v 
\end{align*}
with 
\begin{align} \label{def:diffBC*} 
\mathcal{B}^*[\gamma_-\phi] (t,x) = c_0 \int_{w\cdot n(x)<0} \gamma_-\phi(t,x,w)F(w) |w\cdot n(x)| \d w .
\end{align}
and $\rho^\eps(t,x) = \int_\RR f^\eps(t,x,v) \d v $. Note that $\mathcal{B}^*$ does not depend on $v$ 
because of the simple form of diffuse reflection operator we consider here (constant cross-section).
We then have:
\begin{definition}\label{def:weak}
We say that $f^\eps  \in L^2_{F^{-1}}(\RR_+\times\Omega\times\RR)$ is a weak solution to \eqref{eq:rescaledKinetic}  if for every test function $\phi$ such that $\phi$, $\pa_t \phi$ and $v\pa_x \phi$ are in $L^2_F (\RR_+\times\Omega\times\RR)$ and $\phi$ satisfies the dual boundary condition
\begin{align*}
\gamma_+ \phi = \mathcal{B}^*[\gamma_- \phi]  \quad \mbox{ on } \RR_+\times\RR_+
\end{align*}
we have 
\begin{equation} \label{eq:weakformulation} 
\begin{aligned} 
& \iiint_{ \RR_+\times \Omega\times\RR}  f^\eps \pa_t \phi \d t \d x \d v + \iint_{\Omega\times\RR}  f_{in}(x,v) \phi(0,x,v) \d x \d v   \\
&\qquad \qquad = - \eps^{-2s} \iiint_{ \RR_+\times \Omega\times\RR}  \big[ f^\eps \left( \eps v\pa_x \phi  -  \nu_0 \phi \right) + \nu_0 \rho^\eps F(v) \phi \big] \d t \d x \d v 
\end{aligned} 
\end{equation} 
\end{definition}
Here and in the rest of the paper, we used the notation
$$ L^2_{F^{-1}} ((0,\infty)\times \Omega\times\RR^N)= \left\{f(t,x,v) \, ;\, \int_0^\infty\int_\Omega\int_{\RR^N} |f(t,x,v)|^2 \frac{1}{F(v)}\, dv\, dx\, dt<\infty\right\}$$
and a similar definition for $L^2_{F}((0,\infty)\times \Omega\times\RR^N)$.

The existence of a weak solution in the sense of this definition is discussed, for instance, in \cite{CIP,MM}.

\subsection{The asymptotic diffusion equation} \label{sec:As}
In this section, we recall previous results (in particular our result of \cite{CMP1}) and introduce the asymptotic model.

As already mentioned above, it is now classical that when $\Omega $ is the whole line $\RR$ (or more generally $\Omega=\RR^n$), $f^\eps$ converges to a function
$ \rho(t,x)F(v)$
where $\rho(t,x)$ is the weak solution of a fractional diffusion equation $\pa_t \rho + \kappa \Dels \rho = 0 $.
When $\Omega$ is a subset of $\RR^n$, the diffusion equation must be supplemented by boundary condition. Studying the asymptotic limit of this kinetic equation provides us with the framework to find out physically relevant boundary conditions for fractional diffusion equations. We recall that in the classical diffusion approximation (e.g. when $F$ is a Maxwellian distribution) the limiting equation is the diffusion equation with Neumann boundary conditions.

In \cite{CMP1}, we study the problem \eqref{eq:rescaledKinetic} in dimension $n\geq 1$ when $\Omega$ is the upper half plane. We show that the asymptotic operator (which we denote by $(-\Delta)_N$ since it corresponds to Neumann boundary conditions)  is given by
$$ (-\Delta)_N^s u (x) := - \frac{c_{n,s}}{2s}\int_\Omega \na u(y) \cdot\frac{y-x}{|x-y|^{n+2s}}\, dy,$$
with $c_{n,s} =  \frac{2^{2s} \Gamma\left(\frac{n}{2}+s \right)}{\pi^{n/2} |\Gamma\left(-s \right)|} $ (the constant is chosen here so that when $\Omega = \RR^n$, we recover $(-\Delta)_N^s = (-\Delta)^s$)
which can also be written in divergence form as 
$$ (-\Delta)_N^s u (x) = -\div D_N^{2s-1}[u], \qquad D_N^{2s-1}[u](x) :=  \frac{c_{n,s}}{2s(2s-1)}  \int_\Omega (y-x)\cdot \na u(y)  \frac{y-x}{|y-x|^{n+2s}} \, dy.$$
With these notations, the main result of \cite{CMP1} is:
\begin{theorem}[\cite{CMP1},Theorem 1.1] \label{thm:CMP}
Assume that $F$ satisfies \eqref{def:F} with $s\in(1/2,1)$ and let $\Omega$ be the upper half space 
$ \Omega = \{x\in \RR^n\,;\, x_n >0\}.$
Assume that $f^\eps(t,x,v)$ is a weak solution of \eqref{eq:rescaledKinetic}  in $(0,\infty)\times \Omega\times\RR^n$.

There exists a subsequence $f^{\eps'}$ which converges weakly in $L^\infty(0,\infty; L^2_{F^{-1}} (\Omega\times\RR^n))$  to the function $\rho(t,x) F(v)$ where $\rho(t,x)$ satisfies
\begin{align} \label{eq:weakFracNeumann0}
 \iint_{\RR^+\times\Omega} &\rho(t,x) \Big(  \pa_t \psi(t,x) + \kappa(-\Delta)_N^s  [\psi](t,x) \Big) \d t\d x + \int_\Omega\rho_{in} (x) \psi(0,x) \d x =0 
\end{align}
for all test function $\psi\in C^{1}(0,\infty;H^2(\Omega))$, such that $  (-\Delta)_N^s [\psi] \in L^{2}(\RR_+\times\Omega)$ and  $ D^{2s-1} [\psi] \cdot n =0 $ on $\pa\Omega$.
\end{theorem}

By using the integration by parts formula (see Proposition 3.4 in \cite{CMP1}): 
\begin{equation}\label{eq:symmetry}
\int_{\Omega} \div D^{2s-1}[\vphi] \psi \d x - \int_{\Omega} \vphi \, \div D^{2s-1} [\psi]\d x  = \int_{\pa \Omega} \left[\psi D^{2s-1}[\vphi]\cdot n -
 \vphi D^{2s-1}[\psi]\cdot n \right]\d S(x)
\end{equation}
we see that  \eqref{eq:weakFracNeumann0} is a natural weak formulation for the parabolic boundary value problem
\begin{equation} \label{eq:fracNeumann}
\left\{ \begin{aligned}
& \pa_t \rho - \kappa \, \div D_N^{2s-1} [\rho] = 0 \qquad \mbox{ in } (0,\infty)\times \Omega\\
& D_N^{2s-1} [\rho] \cdot n = 0 \qquad \mbox{ in } (0,\infty)\times\pa  \Omega\\
& \rho(0,x) = \rho_{in}(x) \qquad \mbox{ in } \Omega.
\end{aligned} \right. 
\end{equation}
Using Hille-Yoshida's theorem, we prove in \cite{CMP1} that \eqref{eq:fracNeumann}  is well posed:
\begin{theorem}[\cite{CMP1},Theorem 1.2]\label{thm:evolution}
For all $\rho_{in}\in L^2(\Omega)$ and $\kappa>0$, the evolution problem
\begin{equation}\label{eq:evolution}
\left\{
\begin{array}{ll}
\pa_t \rho - \kappa \, \div D_N^{2s-1}[\rho]=0 & \mbox{ in } (0,\infty)\times \Omega,\\
\rho(0,x) = \rho_{in}(x) & \mbox{ in } \Omega.
\end{array}
\right.
\end{equation}
has a unique  solution $\rho \in C^0([0,\infty);L^2(\Omega))\cap C^1((0,\infty);L^2(\Omega))\cap C^0((0,\infty);D((-\Delta)_N^s))$ where
$$
D((-\Delta)_N^s)=\{u\in H^s(\Omega)\,;\, (-\Delta)_N^s u \in L^2(\Omega), \quad D^{2s-1}[u]\cdot n=0\mbox{ on } \pa\Omega \}.
$$
\end{theorem}
We recall that the space $H^s(\Omega)$ is defined by 
$$ H^s(\Omega)=\left\{ u \in L^2(\Omega)\, ;\,  \int_\Omega\int_\Omega \frac{(u(x)-u(y))^2}{|x-y|^{n+2s}}\d x\d y <\infty\right\}$$
and is equipped with the norm:
$$ \| u\|^2_{H^s} = \int_\Omega  |u(x)|^2\d x+ 
 \int_\Omega\int_\Omega \frac{(u(x)-u(y))^2}{|x-y|^{n+2s}}\d x\d y .$$

Unfortunately, it is not clear that the characterization of $\rho(t,x)$ given by Theorem 
 \ref{thm:CMP} implies that $\rho$ is the unique solution of \eqref{thm:evolution}
provided by Theorem \ref{thm:evolution}.
Indeed, while we can show that the solution  of Theorem \ref{thm:evolution} satisfies  \eqref{eq:weakFracNeumann0}   as in Theorem \ref{thm:CMP}, it does not appear that this formulation is strong enough to yield uniqueness.
The problem is that the condition $\psi\in H^2(\Omega)$ in Theorem \ref{thm:CMP}, which we use in \cite{CMP1} to pass to the limit, is too restrictive to prove uniqueness. In particular,  this condition cannot be deduced from the condition $(-\Delta)_N^s \psi \in L^2(\Omega)$ (or even, as we will see later, from the stronger condition $(-\Delta)_N^s \psi  \in C^\infty(\overline \Omega)$).

\medskip

The aim of the present paper is to show (in dimension $1$) that the convergence result of Theorem~\ref{thm:CMP} can be proved for a different set of test function $\psi$, which allows us to prove that $\rho$ is indeed the unique weak solution of \eqref{eq:fracNeumann} provided by Theorem \ref{thm:evolution}.

\medskip

For future reference, we also recall that 
the key step in the proof of Theorem \ref{thm:evolution} is to show that for all $\lambda>0$, the stationary problem
\begin{equation}\label{eq:Neumanns}
\left\{
\begin{array}{ll}
\lambda u (x)- \div D_N^{2s-1} [u](x)= g(x) & \mbox{ for all  } x\in \Omega, \\
D_N^{2s-1} [u](x) \cdot n(x) =0 & \mbox{ for all } x\in \pa\Omega
\end{array}
\right.
\end{equation}
is well posed in $H^s(\Omega)$.
More precisely, we proved, using Lax Milgram theorem (see Theorem 4.1 and Remark 4.1 in \cite{CMP1}):
\begin{theorem} \label{thm:weak}
For all $\lambda>0$ and $g$ in $L^2(\Omega)$, there exists a unique $u\in D((-\Delta)_N^s)$ solution of \eqref{eq:Neumanns}.
\end{theorem}



\medskip 

\subsection{Main results of the paper}
To state our main result, we introduce the space of test function (for $\beta>0$):
$$
X^{\beta} = 
\left\{
\psi \in C^1_c([0,\infty);C^s(\bO))\, ;\, (-\Delta)_N^s \psi \in L^\infty(0,\infty;C^\beta (\bO)) 
\right\}
$$
and 
$$
X^{\beta}_0 = 
\left\{
\psi \in X^{\beta} \, ; \, D_N^{2s-1} [\psi] \cdot n =0 \mbox{ on } \pa\Omega
\right\}
$$
(we do not indicate the dependence of these spaces on $s$ since $s\in(1/2,1)$ is fixed throughout the paper).
We can now state the main theorem of this paper:
\begin{theorem}\label{thm:main}
Assume that $F$ satisfies \eqref{def:F} with $s\in(1/2,1)$ and assume that the initial condition satisfies, for some constant $C\geq 0$: 
$$ 0\leq f_{in}(x,v) \leq CF(v), \qquad f_{in}\in L^2_{F^{-1}} (\Omega\times\RR)).$$
Let $f^\eps(t,x,v)$ be a weak solution of \eqref{eq:rescaledKinetic}  in $(0,\infty)\times \Omega\times\RR$
in the sense of Definition~\ref{def:weak}. 
Then the function $f^\eps(t,x,v)$ converges weakly in $L^\infty((0,\infty) ; L^2_{F^{-1}} (\Omega\times\RR))$, as $\eps$ goes to $0$, to the function $\rho(t,x) F(v)$ where $\rho(t,x)$ is the {\em unique} function satisfying
\begin{align} \label{eq:weakFracNeumann}
\underset{\RR^+\times\Omega} {\iint} &\rho(t,x) \Big(  \pa_t \psi(t,x) - \kappa(-\Delta)_N^s  [\psi](t,x) \Big) \d t\d x + \underset{\Omega}{\int} \rho_{in} (x) \psi(0,x) \d x =0 
\end{align}
for all test function $\psi \in X^\beta_0$  for some $\beta>0$ and $\kappa=c_{1,s}^{-1}\gamma\nu_0^{1-2s}\Gamma(2s+1)$.
\end{theorem}

Importantly, the uniqueness of the limiting density $\rho(t,x)$ is a new result, which implies that the whole sequence $f^\eps$ (and not just a subsequence) converges. 
This uniqueness was not established in \cite{CMP1} because we required stronger conditions on the test function $\psi$ in order to pass to the limit in \eqref{eq:rescaledKinetic} (see Theorem \ref{thm:CMP}). 
This uniqueness result is of independent interest and can be stated as follows:
\begin{proposition}\label{prop:unique}
Given $\beta>0$ and for all $\rho_{in} \in L^2(\Omega)$, there exists a unique function $\rho(t,x)\in L^\infty(0,\infty;L^2(\Omega))$ satisfying \eqref{eq:weakFracNeumann} 
for all test function $\psi \in X^{\beta}_0$.
\item This solution is also the unique weak solution of \eqref{eq:fracNeumann} provided by Theorem \ref{thm:evolution} and therefore satisfies
$$\rho \in C^0([0,\infty);L^2(\Omega))\cap C^1((0,\infty);L^2(\Omega))\cap C^0((0,\infty);D((-\Delta)_N^s))$$
\end{proposition}


In the proof of Theorem  \ref{thm:main}, we make use of the fact that the condition $(-\Delta)_N^s \psi \in L^\infty(0,\infty;C^\beta (\bO))$ -- which is a natural condition to get the uniqueness of Proposition \ref{prop:unique} -- 
yields some H\"older regularity estimates for $\psi$ (see \eqref{eq:bdRS})  which are exactly what we need to pass to the limit in the proof of Theorem  \eqref{thm:main} (see in particular the proof of Lemma \ref{lem:L}).

While we believe that these H\"older regularity estimates hold in any dimension, we focus on this paper on the one-dimensional case because, as explained below, the operator $(-\Delta)_N^s$ can be written in term of the usual fractional Laplace operator in one dimension, and existing regularity theory \cite{RS} can then be used.
Extending our result to higher dimension would require the development of a regularity theory for the Neumann boundary value problem   \eqref{eq:Neumanns} in higher dimension.

We conclude this section by explaining what makes the one dimensional case so much nicer to work with:
Given a (continuous) function $u$ defined in $\overline \Omega$, we introduce the continuous extension of $u$ by constant:
\begin{equation}\label{eq:extension}
\widetilde u (x) =\begin{cases}
u(0) & \mbox{ if } x\leq 0 \\
u(x) & \mbox{ if } 0\leq x\leq 1 \\
u(1) & \mbox{ if } x\geq 1.
\end{cases}
\end{equation}
We then have:
\begin{align*} 
(-\Delta)_N^s u(x)  
& = - \frac{c_{1,s}}{2s}\int_\Omega u' (y) \frac{y-x}{|x-y|^{1+2s}}\, dy \\
& = - \frac{c_{1,s}}{2s}\int_\RR  \widetilde u '(y) \frac{y-x}{|x-y|^{1+2s}}\, dy\\
& =  c_{1,s} P.V. \int_{\RR} \frac{\widetilde u(x)-\widetilde u(y)}{|x-y|^{1+2s}} \d y.
\end{align*}
that is 
\begin{equation}\label{eq:L1D}
(-\Delta)_N^s u (x) =(-\Delta)^s \widetilde u(x)\qquad \qquad  \mbox{ for all $x\in\Omega$}.
\end{equation}
In particular, we note that  if $u$ is the solution of \eqref{eq:Neumanns} provided by Theorem \ref{thm:weak}, then $\tilde u$ satisfies
$$
\begin{cases}
\lambda \widetilde u +(-\Delta)^s \widetilde u = g & \mbox{ in } \Omega=(0,1)\\
\widetilde u = u(0) & \mbox{ in } (\infty,0)\\
\widetilde u = u(1) & \mbox{ in } (1,\infty)
\end{cases}
$$
and  the regularity theory for the fractional Dirichlet boundary value problem developed for example in \cite{RS,RS2} can be used to study the regularity of $\tilde u$. 
When  $g\in L^\infty(\Omega)$, this gives $\tilde u \in C^s(\RR)$ and this regularity is known to be optimal for the Dirichlet problem. It is not immediately obvious whether this regularity is also optimal for the Neumann boundary value problem or if 
$\tilde u$ inherits better regularity from the Neumann boundary condition. 
We can actually show that this regularity is indeed optimal:
\begin{proposition}\label{prop:reg}
Let $\Omega=(0,1)$ and $g\in L^\infty(\Omega)$, then the solution $u$ of  \eqref{eq:Neumanns} provided by Theorem~\ref{thm:weak} satisfies $u\in C^s(\overline\Omega)$. Furthermore, this regularity is optimal in the sense that there exists $g\in L^\infty(\Omega)$ such that $u(x)\sim x^s$ as $x\to0^+$ and $u(x)\sim (1-x)^s$ as $x\to 1^-$.
\end{proposition}

Note finally that we can also write $(-\Delta)_N^s u(x)  = - \pa_x  D_N^{2s-1}[u]$
where the non local gradient $D_N^{2s-1}$ can also be written, using the extension of $u$, as:
\begin{align}
 D_N^{2s-1}[u](x)
 & :=\frac {c_{1,s}}{2s(2s-1)}  \int_\Omega u'(y) |y-x|^{1-2s}\, dy\nonumber \\
&  =\frac {c_{1,s}}{2s(2s-1)}  \int_\RR \widetilde u'(y) |y-x|^{1-2s}\, dy\nonumber \\
& =\frac {c_{1,s}}{2s }  \int_\RR [\widetilde u(y)-u(x)]  \frac{y-x}{|y-x|^{1+2s}}\, dy. \label{eq:D1D}
\end{align}
The rest of the paper is devoted to the proof of Theorem \ref{thm:main} and Propositions   \ref{prop:unique} and \ref{prop:reg}.

\section{Proof of Theorem \ref{thm:main} and Proposition   \ref{prop:unique}}
\subsection{Construction of the  test functions} \label{subsec:construction}
As in previous work
\cite{Mellet10,AbdallahMelletPuel111,AcevesSchmeiser17,CMP1},  the proof relies on the introduction of an appropriate auxiliary problem: 
\begin{equation}\label{eq:aux0}
\begin{cases}
\nu_0 \phi - \eps v \pa_x\phi  = \nu_0 \psi \qquad \mbox{ in } \Omega\times\RR. \\
\gamma_+ \phi  (t,x,v) = \mathcal B^* [\gamma_- \phi ](t,x)\qquad (x)\in  \Gamma_+.
\end{cases}
\end{equation}
And due to the difficulty of writing an explicit solution  for \eqref{eq:aux0}, we first solve:
\begin{equation}\label{eq:aux1}
\begin{cases}
\nu_0 \phi  - \eps v  \pa_x\phi  = \nu_0 \psi \qquad \mbox{ in } \Omega\times\RR. \\
\gamma_+ \phi  (t,x,v) =\psi(t,x) \qquad (x,v)\in  \Gamma_+.
\end{cases}
\end{equation}
Since we expect to find $ \phi^\eps\sim \psi$ for small $\eps$ and thus $\mathcal B^* [\gamma_- \phi]\sim \psi$, which can be seen as a consequence of the conservation of flux  (namely, the fact that $\mathcal{B}^*[1] = 1$), it is reasonable to expect that the solution of \eqref{eq:aux1} is a good approximation of the solution of \eqref{eq:aux0}. 
We then have:
\begin{proposition}\label{prop:phieps}
Given $\psi \in \mathcal D(\overline \Omega)$,  let $\wpsi $ be the continuous extension of $\psi$ defined as in \eqref{eq:extension}. Then the function
\begin{align*}
\phi (x,v)  = \int_0^\infty  \nu_0 e^{-\nu_0z}\wpsi(x+\eps z v)  \, dz 
\end{align*}
solves \eqref{eq:aux1}. Furthermore, $\phi$ satisfies $\gamma_+ \phi  (x,v) = \mathcal B^* [\gamma_- \phi ](x)$ for $(x,v)\in  \Gamma_+$ (and thus solves \eqref{eq:aux0})  if and only if
\begin{equation}\label{eq:Depspsi}
D^{2s-1}_\eps[\psi] (0) =  D^{2s-1}_\eps[\psi] (1) = 0
\end{equation}
where the operator $D^{2s-1}_\eps$ is defined by (\ref{eq:defD}) below.
\end{proposition}
\begin{proof} 
We easily check that $\phi(x,v)$ given by Proposition \ref{prop:phieps} solves  \eqref{eq:aux1}.
Indeed, we have:
$$ \eps v  \pa_x\phi = \int_0^\infty \nu_0 e^{-\nu_0z} \frac{d}{dz} \left[  \wpsi(x+\eps z v)\right] \, dz =  \int_0^\infty  \nu_0^2 e^{-\nu_0z} \wpsi(x+\eps z v)\, dz -\nu^0\psi(x)  = \nu_0(\phi-\psi)$$
and if $(x,v)\in\Gamma_+$, for instance if $x=0$ and $v<0$, then
$$
\phi (0,v)  = \int_0^\infty  \nu_0 e^{-\nu_0z}\wpsi(\eps z v)  \, dz  = \int_0^\infty  \nu_0 e^{-\nu_0z}\psi(0)  \, dz =\psi(0).
$$
Next, 
%
we note that $\phi$ satisfies $\gamma_+ \phi  (x,v) = \mathcal B^* [\gamma_- \phi ](x)$ on $\Gamma_+$
if and only if $\mathcal B^* [\gamma_- (\phi-\psi)] = 0$ on $\Gamma_+$, which, using \eqref{def:diffBC*} and the fact that $\phi(x,v)-\psi(x)=0$ on $\Gamma_+$, is equivalent to
$$
\int_{\RR} v F(v) [\phi(x,v)-\psi(x)] \, dv = 0 \qquad \mbox{ on }  \pa\Omega.
$$
The result then follows by introducing the operator
\begin{align}
D^{2s-1}_\eps[\psi](x) &: = \eps^{1-2s} \int_{\RR} v F(v) [\phi(x,v)-\psi(x)] \, dv\nonumber   \\
& = \eps^{1-2s} \int_{\RR} \int_0^\infty \nu_0  e^{-\nu_0z}    v F(v)    [ \wpsi(x+\eps z v)-\psi(x) ] \, dz  \, dv. \label{eq:defD}
\end{align}
\end{proof}

Since we want to use the function $\phi(x,v)$ as a test function in \eqref{eq:weakformulation}, we need $\phi$ to satisfy the condition $\gamma_+ \phi  (x,v) = \mathcal B^* [\gamma_- \phi ](x)$.
We cannot require a given function $\psi$ to satisfy \eqref{eq:Depspsi} since this condition depends on $\eps$.
But we can approximate a given test function $\psi$ 
by a function $\psi^\eps$ satisfying \eqref{eq:Depspsi}. 
To that end, we consider a smooth function $\chi$ satisfying 
\begin{align} \label{def:chi}  
\chi \in C^\infty(\bO), \quad \chi(0)=1, \quad \supp (\chi)\subset [0,1/2),\quad 0\leq  \chi(x) \leq 1
\end{align}
(these conditions guarantee that $ D_N^{2s-1} [\chi](0)\neq D_N^{2s-1}[\chi](1)$).
We then have
\begin{proposition} \label{prop:psieps} 
	Given $\psi \in \mathcal{D}(\overline \Omega)$ we defined $\psi^\eps$ as 
	\begin{align} \label{def:psieps} 
	\psi^\eps(x) = \psi(x) + \le_0  \chi(x) + \le_1 \chi(1-x) \qquad \forall x\in\Omega 
	\end{align} 
	with 
	\begin{equation} \label{eq:le}
	\left\{ \begin{aligned}
	 \le_0 = \frac{-D^{2s-1}_\eps [\chi](0) D^{2s-1}_\eps[\psi](0) + D^{2s-1}_\eps[\chi](1)D^{2s-1}_\eps[\psi](1)}{(D^{2s-1}_\eps[\chi](0))^2-(D^{2s-1}_\eps[\chi](1))^2} ,\\[8pt]
	 \le_1 = \frac{ -D^{2s-1}_\eps[\chi](0) D^{2s-1}_\eps[\psi](1)   + D^{2s-1}_\eps[\chi](1)D^{2s-1}_\eps[\psi](0)}{(D^{2s-1}_\eps[\chi](0))^2-(D^{2s-1}_\eps[\chi](1))^2}
	\end{aligned} \right.
	\end{equation}
	where $D^{2s-1}_\eps$ is defined by \eqref{eq:defD}. Then 
	\begin{align}
	D^{2s-1}_\eps [\psi_\eps ](0) =D^{2s-1}_\eps [\psi_\eps] (1) = 0 \label{eq:psiepsbc}
	\end{align}
\end{proposition}
\begin{proof}
We note that
\begin{align*}
 D^{2s-1}_\eps[\chi(1-x)] (0) 
& = D^{2s-1}_\eps[\chi] (1) 
\end{align*}
so the linearity of the operator $ D^{2s-1}_\eps$ and the choice of $\le_0$ and $\le_1$ implies
\eqref{eq:psiepsbc}.
Note that  we will prove later that $D_\eps^{2s-1}[\chi](x)$ converges to $\gamma_0 \frac{2s}{c_{1,s}}D_N^{2s-1}[\chi](x)$ (see Lemma \ref{lem:DepsCV}). So the fact that $ D_N^{2s-1} [\chi](0)\neq D_N^{2s-1}[\chi](1)$ and Lemma \ref{lem:DepsCV} imply that the denominator in \eqref{eq:le} does not vanish for $\eps$ small enough.

\end{proof}

We now have all the tools needed to set up the proof of our main result:
for a given test function $\psi(t,x)$ in $\mathcal D([0,\infty)\times \overline \Omega)$, we consider $\psi^\eps(t,x)$ given by Proposition~\ref{prop:psieps}. Then  the function 
\begin{equation}\label{lephieps}
\phi^\eps (t,x,v)  =    \int_0^\infty  \nu_0 e^{-\nu_0z} \wpsi^\eps(t,x+\eps z v)  \, dz
\end{equation}
solves \eqref{eq:aux0} (see Proposition \ref{prop:phieps}) and by taking $\phi^\eps$ as a test function in
\eqref{eq:weakformulation}, we  find:
\begin{align}
&  \iiint_Q f^\eps \pa_t \phi^\eps \d t \d x \d v + \iint_{\Omega\times\RR}  f_{in}(x,v) \phi^\eps(0,x,v) \d x \d v \nonumber \\
&\qquad\qquad\qquad= - \eps^{-2s} \iiint_Q \big[ f^\eps \left( \eps v\pa_x \phi^\eps  -  \nu_0 \phi ^\eps \right) + \nu_0 \rho_\eps F(v) \phi^\eps \big] \d t \d x \d v \nonumber\\
&\qquad\qquad\qquad= - \eps^{-2s} \iiint_Q \big[ -\nu_0 f^\eps  \psi^\eps  + \nu_0 \rho_\eps F(v) \phi^\eps \big] \d t \d x \d v \nonumber \\
&\qquad\qquad\qquad= - \eps^{-2s} \iiint_Q \rho_\eps F(v)  \nu_0  \big[ \phi^\eps-  \psi^\eps   \big] \d t \d x \d v \nonumber \\
&\qquad\qquad\qquad= -  \iint_{\Omega\times\RR}  \rho_\eps\L^\eps[\psi^\eps] (x) \d t \d x .\label{eq:weakeps}
\end{align}
where we used the fact that $\int_\RR  f^\eps\, dv = \rho^\eps = \int_\RR  \rho^\eps F(v)\, dv$ 
and we defined the following operator (for any test function $\psi$ and with $\phi$ defined by Proposition \ref{prop:phieps}):
\begin{align}
\L^\eps[\psi ] (x) & : = \eps^{-2s} \int_\RR   \nu_0 F(v) \big[ \phi(x,v) -  \psi (x) \big]\, dv \nonumber \\
& = \eps^{-2s} 
\int_\RR   
 \int_0^\infty \nu_0^2  e^{-\nu_0z}  F(v)  [\wpsi(x+\eps z v)-\psi(x)] \, dz \, dv .\label{def:Leps}
\end{align}

The proof of Theorem \ref{thm:main} now consists in passing to the limit in \eqref{eq:weakeps}, which requires, in particular, to show that for appropriate $\psi$, the function $\L^\eps[\psi^\eps ] $ converges (strongly in $L^1$) to $ \kappa (-\Delta)^s [\wpsi]=\kappa (-\Delta)^s_N [\psi]$.

In the section below, we first derive simpler formulas for $\L^\eps$ and $D^{2s-1}_\eps$. These formulas will then be used to prove the needed convergence results.


\subsection{Reformulation of the  operators $\L^\eps$ and $D^{2s-1}_\eps$}
After a simple change of variable, we find the following formula for the operator $\L^\eps[\psi ] $, defined by \eqref{def:Leps}:
$$
\mathcal L^\eps[ \psi] (x)  =  \int_{\Omega} F_1^\eps (y-x)   [\wpsi(y)-\psi(x)]  dy 
$$
with
\begin{equation} \label{def:F1}  
F_1(v) = \int_0^{+\infty} \nu_0^2 e^{-\nu_0 \tau}  \tau^{-1} F\left(\tau^{-1} v\right) \d \tau, \quad F^\eps_1(v) =\frac{1}{\eps^{1+2s}} F_1\left(\frac{v}{\eps}\right).
\end{equation}
Similarly, the operator $D^{2s-1}_\eps[\psi] $ introduced in \eqref{eq:defD} can be written as:
\begin{align}
D^{2s-1}_\eps[\psi]
& = \int_{\RR}  (y-x) F_0^\eps(y-x)  [ \widetilde \psi(y)-\psi(x)]\, dy \label{eq:D}
\end{align}
with
\begin{equation} \label{def:F0}  
F_0(v) = \int_0^{+\infty} \nu_0 e^{-\nu_0 \tau}  \tau^{-2} F\left(\tau^{-1} v\right) \d \tau, \quad F^\eps_0(v) =\frac{1}{\eps^{1+2s}} F_0\left(\frac{v}{\eps}\right).
\end{equation}

The introduction of the functions $F_0$ and $\bF$ allow us to eliminate the variable $z$ from the definition of $D^{2s-1}_\eps$ and $\L^\eps$. 
Of course, their behavior for large $v$ is related to that of $F$. More precisely, we have the following Lemma:
\begin{lemma} \label{lem:estimatesFi} 
There exists a constant $C>0$ such that the distributions $F_0$ and $F_1$ given by \eqref{def:F0} and \eqref{def:F1} satisfy
\begin{align*}
F_0(z) \leq C \min\left( \frac{ 1}{|z|^{1+2s}}, \frac{1}{|z|}\right), \qquad F_1(z) \leq C \min\left( \frac{1}{|z|^{1+2s}} , |\ln (z)|\right),  \mbox{ for all $z\in\RR$}
\end{align*}
and
\begin{align*}
\left| F_i(z) - \frac{\gamma_i}{|z|^{1+2s}}\right| \leq \frac{C}{|z|^{1+4s}} \qquad\mbox{ for all } |z|\geq 1
\end{align*} 
where $\gamma_i = \gamma \nu_0^{1-2s} \Gamma(2s+i)$, with $\gamma$ the constant of $F$ in \eqref{def:F}.

\end{lemma}
\begin{proof}
We first note that, for $i=0,1$ we can write $F_i$ as
\begin{align*}
F_i(z) = \int_0^{+\infty} \nu_0^{1+i} e^{-\nu_0 \tau} \tau^{i-2} F\left(\frac{z}{\tau}\right) \d \tau . 
\end{align*}
For the first estimate, using \eqref{def:F} we write on the one hand
\begin{align*}
\int_0^z \nu_0^{1+i} e^{-\nu_0 \tau} \tau^{i-2} F\left(\frac{z}{\tau}\right) \d \tau 
&\leq C \int_0^z   e^{-\nu_0\tau} \tau^{i-2} \left( \frac{ \tau^{1+2s}}{|z|^{1+2s}} + \frac{\tau^{1+4s}}{|z|^{1+4s}} \right) \d \tau\\
&\leq \frac{C}{|z|^{1+2s}} \min\left(1,  |z|^{2s+i} \right)  + \frac{C}{|z|^{1+4s}} \min\left(1,  |z|^{4s+i} \right)  \\
&\leq C  \min\left( \frac{1 }{|z|^{1+2s}}, \frac{1}{|z|^{1-i}} \right)  
\end{align*}
and on the other hand 
\begin{align*}
\int_z^{+\infty} \nu_0^{1+i} e^{-\nu_0 \tau} \tau^{i-2} F\left(\frac{z}{\tau}\right) \d \tau 
& \leq\|F\|_{L^\infty} \int_z^{+\infty} \nu_0^{1+i} e^{-\nu_0 \tau} \tau^{i-2}   \d \tau 
\end{align*}
with
$$\int_z^{+\infty} \nu_0^{1+i} e^{-\nu_0 \tau} \tau^{i-2}   \d \tau =\int_z^{+\infty} \nu_0 e^{-\nu_0 \tau} \tau^{-2}   \d \tau \leq C\max\left(e^{-\nu_0z} , \frac{1}{z}\right) \leq \frac C z \quad \mbox{ for $i=0$}$$
$$\int_z^{+\infty} \nu_0^{1+i} e^{-\nu_0 \tau} \tau^{i-2}   \d \tau =\int_z^{+\infty} \nu_0^{2} e^{-\nu_0 \tau} \tau^{-1}   \d \tau \leq C\max\left(e^{-\nu_0z} , \ln(z)\right)\leq C\ln(z)  \quad \mbox{ for $i=1$}$$
The first estimates follow.

\medskip


To prove the second estimates, we use the formula $\int_0^\infty \nu_0^{1+i} \tau^{2s+i-1} e^{-\nu_0 \tau} \d \tau = \nu_0^{1-2s} \Gamma(2s+i)$ to get:
\begin{align*}
&\left| F_i(z) - \frac{\gamma_i}{|z|^{1+2s}} \right|  
\leq \int_0^\infty \nu_0^{1+i} e^{-\nu_0 \tau} \left| \tau^{i-2} F\left(\frac{z}{\tau}\right) - \frac{\gamma\tau^{2s+i-1}}{|z|^{1+2s}} \right| \d \tau\\
&\hspace{60pt} \leq \int_0^z \nu_0^{1+i} e^{-\nu_0 \tau} \tau^{4s+i-1} \frac{C}{|z|^{1+4s}}  \d \tau 
+
\int_z^\infty   e^{-\nu_0 \tau}  \nu_0^{1+i} \left( \tau^{i-2} \|F\|_{L^\infty} + \frac{\gamma\tau^{2s+i-1}}{|z|^{1+2s}} \right) \d \tau\\
&\hspace{60pt} \leq \frac{C}{|z|^{1+4s}}  \int_0^\infty e^{-\nu_0 \tau} \tau^{4s+i-1} \d \tau 
+
C\int_z^\infty  e^{-\nu_0 \tau} \tau^{i-2}    \d \tau
+
\frac{C}{|z|^{1+2s}}  \int_z^\infty  e^{-\nu_0 \tau} \tau^{2s+i-1} \d \tau
\end{align*}
and the result follows.

\end{proof}

\subsection{Convergence of the operator $D_\eps^{2s-1}$ and $\L^\eps$ } 
In order to pass to the limit in \eqref{eq:weakeps}, we need to show that $\L^\eps[\psi^\eps]$ converges strongly in $L^1$ when $\psi\in X^\beta_0$.  The key result of this section is the following proposition:
\begin{proposition}\label{prop:Leps}
Given $\psi\in C^s(\bO)$
such that $(-\Delta)_N^s[\psi] \in C^\beta (\bO)$ for some $\beta>0$ and satisfying  $D^{2s-1}_N[\psi]=0$ on $\pa\Omega$, let $\psi^\eps$ be defined as in \eqref{def:psieps}. 
Then
\begin{align*}
\Leps [\psi^\eps] \to \L[\psi] := - \kappa (-\Delta)^s [\wpsi]  \mbox{ strongly in } L^1(\Omega)
\end{align*}
with $\kappa=c_{1,s}^{-1}\gamma_1$. 
\end{proposition}
This result implies in particular the convergence of $\Leps [\psi^\eps(t,\cdot)]$ for all $t$ whenever  $\psi\in X^\beta_0$.
Its proof will follow from the following two lemmas:
\begin{lemma} \label{lem:DepsCV}
Let $\psi\in C^\alpha(\bO)$ with $\alpha>2s-1$, then
\begin{align}
 D_\eps^{2s-1}[\psi](x) \rightarrow \gamma_0 \frac{2s}{c_{1,s}}D_N^{2s-1}[\psi](x) \mbox{ uniformly in } \bO.
\end{align}
In particular, if $\psi$ satisfies $D^{2s-1}_N[\psi]=0$ on $\pa\Omega$ then the constants defined by \eqref{eq:le} satisfy
$$
\lim_{\eps\to 0} \le_0 = \lim_{\eps\to0}\le_1 =0.$$
\end{lemma}
and
\begin{lemma}\label{lem:L}
Assume that $\psi\in C^s(\bO)$ and $(-\Delta)_N^s[\psi] \in C^\beta (\bO)$ for some $\beta>0$. Then
\begin{align*}
\Leps [\psi] \to \L[\psi] = -\kappa(-\Delta)^s  [\wpsi]  \mbox{ strongly in } L^1(\Omega)
\end{align*}
with $\kappa=\frac{\gamma_1}{c_{1,s}}$. 
\end{lemma}

\begin{proof}[Proof of Proposition \ref{prop:Leps}]
In view of \eqref{def:psieps}, we have
$$\L^\eps[	\psi^\eps] (x)= \L^\eps[ \psi](x) + \le_0  \L^\eps[ \chi ] + \le_1 \L^\eps[ \chi(1-\cdot)](x) .
$$
Lemma \ref{lem:L} implies the convergence of $\L^\eps[ \psi]$, $\L^\eps[ \chi ] $ and $\L^\eps[ \chi(1-\cdot)]$ in $L^1$ and Lemma \ref{lem:DepsCV} gives $
\lim_{\eps\to 0} \le_0 = \lim_{\eps\to0}\le_1 =0.$
The result follows.
\end{proof}

The rest of this section is devoted to the proof of the two lemma.

\begin{proof}[Proof of Lemma \ref{lem:DepsCV}]
We write
$$
 D_\eps^{2s-1}[\psi](x)- \gamma_0 \frac{2s}{c_{1,s}}D_N^{2s-1}[\psi](x)
=  \int_{\RR}    [ \widetilde \psi(x+y)-\psi(x)] y \left[ F_0^\eps(y) - \frac{\gamma_0}{|y|^{1+2s}}\right]\, dy .
$$
Lemma \ref{lem:estimatesFi} gives the following bounds: 
$$
 |y| F^\eps_0\left(y \right)  \leq \frac{C}{|y|^{2s}} \quad \forall y \in\RR,\qquad  \mbox{ and } \qquad
  \left|  |y| F^\eps_0\left(y \right)   - \frac{\gamma_0}{|y|^{2s}}\right| \leq \frac{C\eps^{2s}}{|y|^{4s}} \quad \forall y  >\eps.
$$
We thus have, using the $C^\alpha(\bO)$ regularity of $\psi$, with $\alpha>2s-1$:
\begin{align*}
\left|D_\eps^{2s-1}[\psi](x)- \gamma_0 \frac{2s}{c_{1,s}}D_N^{2s-1}[\psi](x)\right|  & \leq 
C \int_{|y|\leq \eps}   \frac{| \widetilde \psi(x+y)-\psi(x)|}{|y|^{2s}} \, dy
 +C \eps^{2s}\int_{|y|\geq \eps}  \frac{| \widetilde \psi(x+y)-\psi(x)|}{|y|^{4s}}\, dy\\
 & \leq C \eps^{1+\alpha-2s}  + C \eps^{2s} (1+\eps^{1+\alpha-4s})   \\
 & \leq C[ \eps^{2s}+ \eps^{1+\alpha-2s} ].
 \end{align*}
\end{proof}



\begin{proof}[Proof of Lemma \ref{lem:L}]
As noted in the introduction, a crucial observation in this proof is the fact that the condition $(-\Delta)_N^s[\psi] \in C^\beta (\bO)$ implies some H\"older regularity for $\psi$. 
Indeed, since $(-\Delta)^s_N [\psi]= (-\Delta)^s [\wpsi]$, we can use the regularity  theory developed in  \cite{RS} to get the following estimate (we use here the notation of   \cite{RS} for the H\"older norms):
$$
\| \psi\|_{\beta+2s}^{(-s)} \leq C (\| \psi\|_{C^s} + \| (-\Delta)^s_N [\psi] \| _\beta^{(s)} )
$$
where 
$$
\| \psi\|_{\beta+2s}^{(-s)}: = \sup_\Omega d_x ^{-s} u(x) + \sup_\Omega d_x ^{1-s} u'(x) + \sup_{(x,y) \in \Omega^2} d_{x,y}^{\beta+s} \frac{ |\psi'(x) - \psi'(y)|} {|x-y|^{\beta+2s-1}}.
$$
and 
$$\| g\|  _\beta^{(s)} :=    \sup_{(x,y) \in \Omega^2} d_{x,y}^{\beta+s} \frac{ |g(x) - g(y)|} {|x-y|^{\beta}}.
$$
with
$$ d_x = \mathrm{dist}(x,\pa\Omega) , \qquad d_{x,y} = \min(d_x,d_y).$$
We deduce that for any $\psi$ satisfying the conditions of Lemma \ref{lem:L},  we have
\begin{equation}\label{eq:bdRS}
\sup_\Omega d_x ^{-s} u(x) + \sup_\Omega d_x ^{1-s} u'(x) + \sup_{(x,y) \in \Omega^2} d_{x,y}^{\beta+s} \frac{ |\psi'(x) - \psi'(y)|} {|x-y|^{\beta+2s-1}} \leq C.
\end{equation}

We now  recall that 
$$
\Leps [\psi ](x) =  \int_\RR \big( \wpsi (x+y) - \wpsi (x) \big) F^\eps_1(y) \d y
$$
where Lemma \ref{lem:estimatesFi} gives (recall that $ F^\eps_1(y)  = \eps^{-1-2s} F_1(y/\eps) $):
\begin{align*}
F^\eps_1(y) \leq \frac{\gamma_1}{|y|^{1+2s}} \quad \forall y\in\RR\quad \mbox{ and } \quad \quad \left| F_1^\eps(y) - \frac{\gamma_1}{|y|^{1+2s}}\right| \leq \frac{C \eps^{2s}}{|y|^{1+4s}} \quad \forall |y|>\eps.
\end{align*}
For $\alpha \in (0,1)$ (to be chosen later), we thus have (for $x\in \Omega)$):
\begin{align*}
\Leps [\psi](x) - \L[\psi](x)  
& = \int_\RR \big( \wpsi (x+y) - \wpsi(x) \big) \left[F^\eps_1(y) -  \frac{\gamma_1}{|y|^{1+2s}} \right] \d y \\
& = \int_{|y|\leq \eps^\alpha}  \big[ \wpsi (x+y) - \wpsi(x) - \wpsi '(x) y \big] \left[F^\eps_1(y) -  \frac{\gamma_1}{|y|^{1+2s}}\right] \d y \\
&\quad  + \int_{|y|\geq \eps^\alpha} \big[ \wpsi (x+y) - \wpsi(x) \big] \left[F^\eps_1(y) -  \frac{\gamma_1}{|y|^{1+2s}} \right] \d y 
\end{align*}
which yields
$$
| \Leps [\psi](x) - \L[\psi](x)  | \leq \int_{|y|\leq   \eps^\alpha}  \big| \wpsi (x+y) - \wpsi(x) - \wpsi '(x) y \big|\frac{C}{|y|^{1+2s}}  \d y + \int_{|y|\geq  \eps^\alpha} \big|\wpsi (x+y) - \wpsi(x) \big|\frac{C \eps^{2s}}{|y|^{1+4s}} \d y .
$$
The second term is clearly bounded by $\| \psi\|_\infty \eps^{2s-4s \alpha} $, so we can write
$$
\int_\Omega | \Leps [\psi](x) - \L[\psi](x)  |\, dx 
\leq 
\int_\Omega \int_{|y|\leq  \eps^\alpha}  \big| \wpsi (x+y) - \wpsi(x) - \wpsi '(x) y \big|\frac{C}{|y|^{1+2s}}  \d y\, dx + \| \psi\|_\infty \eps^{2s(1-2\alpha)}
$$
and we write the integral in the right hand side as $J^\eps_1+J^\eps_2$ with
$$
J_1^\eps=
\int_\Omega \int_{|y|\leq \eps^\alpha}  \big| \wpsi (x+y) - \wpsi(x) - \wpsi '(x) y \big|\frac{C}{|y|^{1+2s}}  1_{(x+y)\notin \Omega} \d y\, dx $$
$$
J_2^\eps=
\int_\Omega \int_{|y|\leq  \eps^\alpha}  \big| \wpsi (x+y) - \wpsi(x) - \wpsi '(x) y \big|\frac{C}{|y|^{1+2s}}  1_{(x+y)\in \Omega} \d y\, dx .
$$
In order to bound $J_1^\eps$, we note that if $(x+y)\notin \Omega$ then $|y|\geq d_x$. Using \eqref{eq:bdRS}, we deduce
\begin{align*}
 J_1^\eps 
& \leq \int_{d_x\leq  \eps^\alpha} \int_{d_x\leq |y|\leq  \eps^\alpha}  \big[ |\wpsi (x+y) - \wpsi(x)| + | \wpsi '(x) | y \big]\frac{C}{|y|^{1+2s}} 1_{(x+y)\notin \Omega}  \d y\, dx \\
& \leq C \int_{d_x\leq  \eps^\alpha} \int_{d_x<|y|\leq  \eps^\alpha}  \big[ d_x^s  + d_x^{s-1} |y| \big]\frac{C}{|y|^{1+2s}}  \d y\, dx \\
& \leq C \int_{d_x\leq  \eps^\alpha}  d_x^{-s}  \, dx \leq C\eps^{\alpha(1-s)}
\end{align*}
For $J_2^\eps$, we first notice that when $x+y\in\Omega$ we have (using \eqref{eq:bdRS}):
$$|\wpsi (x+y) - \wpsi(x) - \wpsi '(x) y |= \left| \int_0^1 \psi'(x+\tau y) y - \psi'(x)y\, \d \tau \right| 
\leq C \frac{y^{\beta+2s}}{d_{x,x+y}^{\beta+s}}$$
and so
\begin{align*}
J_2^\eps 
& \leq C
\int_\Omega \int_{|y|\leq \eps^\alpha}   \frac{y^{\beta-1 }}{d_{x,x+y}^{\beta+s}}1_{(x+y)\in \Omega}  \d y\, dx \leq C
 \int_{|y|\leq \eps^\alpha} y^{\beta-1 }    \int_\Omega d_{x,x+y}^{-\beta-s} 1_{(x+y)\in \Omega}  \, dx \d y 
\end{align*}
As long as $\beta+s<1$, we have $ \int_\Omega d_{x,x+y}^{-\beta-s} 1_{(x+y)\in \Omega}  \, dx <\infty$ and so
$$ J_2^\eps \leq  C
 \int_{|y|\leq  \eps^\alpha} y^{\beta-1 }   \d y \leq C \eps^{\beta\alpha}
$$

We have thus proved (provided $0<\beta<1-s$):
$$
\int_\Omega | \Leps [\psi](x) - \L[\psi](x)  |\, dx 
\leq C[ \eps^{2s(1-2\alpha)} + \eps^{\beta\alpha} +C\eps^{\alpha(1-s)}]
$$
and the result follows by taking $\alpha \in (0,1/2)$.
\end{proof}

\subsection{Convergence of $\phi^\eps$}
Finally, in order to pass to the limit in the remaining terms in \eqref{eq:weakeps}, we need the convergence of $\phi^\eps$ and $\pa_t\phi^\eps$:
\begin{lemma} \label{lem:CVphieps}
Consider $\psi \in C^0_c([0,\infty);C^\alpha(\bO))$ with $\alpha>2s-1$ such that $D^{2s-1}_N[\psi]=0$ on $\pa\Omega\times (0,\infty)$. Then 
\begin{align*}
\lim_{\eps \to 0} \iiint_{\RR_+\times\Omega\times\RR} |\phi^\eps -  \psi^\eps |^2 F(v) \d v \d x \d t = 0 .
\end{align*}
If  $\psi \in C^1_c([0,\infty);C^\alpha(\bO))$, with $\alpha>2s-1$ then 
\begin{align*}
\lim_{\eps \to 0} \iiint_{\RR_+\times\Omega\times\RR} |\pa_t \phi^\eps -  \pa_t \psi^\eps |^2 F(v) \d v \d x \d t = 0 
\end{align*}

\end{lemma}
\begin{proof}
First, we note that
$$
| \phi^\eps (t,x,v)-\psi^\eps(t,x)|^2  \leq  \int_0^\infty  \nu_0 e^{-\nu_0z}|\wpsi^\eps(t,x+\eps z v) -\psi^\eps(t,x)|^2 \, dz 
$$
and so (with $T$ such that $\psi(t)=0$ for $t\geq T$):
\begin{align*}
&\iiint_{\RR_+\times\Omega\times\RR} |\phi^\eps-  \psi^\eps |^2 F(v) \d v \d x \d t \\
&\qquad  \leq \iiint_{(0,T)\times\Omega\times\RR}   \int_0^\infty  \nu_0 e^{-\nu_0z}|\wpsi(t,x+\eps z v) -\psi^\eps(t,x)|^2 F(v) \, dz \d v \d x \d t \\
&\qquad  \leq \iint_{(0,T)\times\RR}   \int_0^\infty  \int_\Omega |\wpsi^\eps(t,x+\eps z v) -\psi^\eps(t,x)|^2 \, dx F(v)\nu_0 e^{-\nu_0z}\ \, dz \d v  \d t 
\end{align*}
Then, we note that (recall that $\le_0,\le_1\to 0$ by Lemma \ref{lem:DepsCV}):
$$\lim_{\eps\to 0}\int_\Omega |\wpsi^\eps(t,x+\eps z v) -\psi^\eps(t,x)|^2 \, dx =0 \quad \mbox{ for all } t,v,z$$
and
$$\int_\Omega |\wpsi(t,x+\eps z v) -\psi^\eps(t,x)|^2 \, dx F(v)\nu_0 e^{-\nu_0z}\leq \| \psi\|_{L^\infty}^2  F(v)\nu_0 e^{-\nu_0z} \in L^1 ((0,T)\times\RR\times\RR)$$
Lebesgue dominated convergence theorem implies the result.

The second limit is proved similarly (note that $t$ is a parameter).
\end{proof}

\subsection{Proof of Theorem \ref{thm:main}}
\begin{proof}[Proof of Theorem \ref{thm:main}]
We are now ready to prove our main result.
\paragraph{A priori estimates.} We have the following classical lemma:
\begin{lemma} \label{lem:apriori}
Let $f_{in}$ be in $L^2_{F^{-1}}(\Omega\times\RR)$. The weak solution $f^\eps$ of \eqref{eq:rescaledKinetic}
 is bounded in $ L^\infty(0,\infty; L^2_{F^{-1}}(\Omega\times\RR^N))$ and 
 satisfies, up to a subsequence
\begin{equation}\label{eq:flim}
f^\eps \rightarrow \rho(t,x) F(v) \quad \text{ weakly in } L^\infty(0,\infty; L^2_{F^{-1}}(\Omega\times\RR))
\end{equation}
where $\rho(t,x)$ is the weak limit of $\rho^\eps (t,x) = \int_{\RR} f^\eps \d v $.
Assume furthermore that $f_{in}(x,v)\leq CF(v)$ for some constant $C$.
Then  $f^\eps(t,x,v) \leq CF(v)$ and 
\begin{equation}\label{eq:rholim} \rho^\eps(t,x) \rightharpoonup \rho(t,x) \quad  L^\infty(\RR_+\times\Omega\times\RR)
 \star-\mbox{weak.} 
\end{equation}
\end{lemma}

\begin{proof}
We do not prove the first part of the lemma which is classical (see for instance Lemma 2.1 in \cite{CMP1}).
\medskip

For the second part, we note that when $f_{in}(x,v)\leq CF(v)$, the function $(t,x,v)\mapsto CF(v)-f^\eps(t,x,v)$ is a solution of \eqref{eq:rescaledKinetic} with non-negative initial data and thus is thus non-negative for all time. This implies $f^\eps(t,x,v)\leq CF(v)$ and so $\rho^\eps(t,x)\leq C.$
\end{proof}

\paragraph{Convergence to a solution of the asymptotic problem.}
Given a test function $\psi\in X^\beta_0$, 
we can now pass to the limit in the weak formulation \eqref{eq:weakeps}, which we recall here:
$$
 \iiint_Q f^\eps \pa_t \phi^\eps \d t \d x \d v + \iint_{\Omega\times\RR}  f_{in}(x,v) \phi^\eps(0,x,v) \d x \d v  = -  \iint_{\Omega\times\RR}  \rho_\eps\L^\eps[\psi^\eps] (x) \d t \d x .
$$
For any subsequence along which \eqref{eq:flim} holds,  Proposition \ref{prop:Leps}, Lemma \ref{lem:CVphieps}  (both of which apply since $\psi\in X^\beta_0$) and \eqref{eq:rholim}  allow us to take the limit, proving that the limiting density $\rho(t,x)$ satisfies \eqref{eq:weakFracNeumann}.

The fact that the whole sequence converges then follows from the uniqueness of the limit $\rho$ given by Proposition \ref{prop:unique} which we prove below.
This completes the proof of Theorem \ref{thm:main}.
\end{proof}

\subsection{ Proof of Proposition \ref{prop:unique}}
\begin{proof}[Proof of Proposition \ref{prop:unique}]

To prove the existence, we simply have to show that the weak solution $u(t,x)$ provided by Theorem \ref{thm:evolution} satisfies \eqref{eq:weakFracNeumann} for appropriate test functions.
We recall that $u \in C^0([0,\infty);L^2(\Omega))\cap C^1((0,\infty);L^2(\Omega))\cap C^0((0,\infty);D((-\Delta)_N^s))$
and
that in dimension one, the condition $(-\Delta)_N^s u\in C^0((0,\infty); L^2(\Omega))$ implies that $(t,x)\mapsto D_N^{2s-1}[u](t,x) $ is continuous (so the Neumann condition is satisfied in the classical sense).
Using integration by parts and \eqref{eq:symmetry}, it is easy to check that $u$ satisfies
$$
\underset{\RR^+\times\Omega} {\iiint}  u  \big( \pa_t \psi - \kappa  \,(-\Delta)_N^s \psi \big) \d t\d x + \underset{\Omega}{\iint} \rho_{in} (x) \psi(0,x) \d x =0 
$$
for all test function $\psi \in X^\beta_0$.

\medskip

Let now $\rho_1(t,x)$ and $\rho_2(t,x)$ be two functions satisfying \eqref{eq:weakFracNeumann} for appropriate test functions.
The function $\bar\rho =\rho_1-\rho_2 \in L^\infty(0,\infty;L^2(\Omega))$ satisfies
\begin{equation} \label{eq:weakFracNeumannuni}
\underset{\RR^+\times\Omega} {\iiint} \bar \rho(t,x) \Big(  \pa_t \psi  -\kappa \,(-\Delta)_N^s \psi  \Big)(t,x) \d t\d x =0 
\end{equation}
for all test function $\psi \in X^\beta_0$.
Given a smooth test function $g\in \mathcal D(\Omega)$ and $\lambda >0$, we let $\phi(x)$ be the weak solution of 
$$
\left\{
\begin{array}{ll}
\lambda \phi + \kappa\, (-\Delta)_N^s \phi = g & \mbox{ in } \Omega, \\
D_N^{2s-1} [\phi](x) \cdot n(x) =0 & \mbox{ for all } x\in \pa\Omega
\end{array}
\right.
$$
given by Theorem \ref{thm:weak} and we define
$$\psi(t,x) = e^{-\lambda t } \phi(x).$$
We need to check that we can take this function $\psi$ as test function in \eqref{eq:weakFracNeumannuni}: 

First, the maximum principle (for the Neumann boundary value problem) implies that $\| \phi\|_{L^\infty} \leq C\| g\|_{L^\infty}$. 
Next, we note that the extension $\tilde \phi$ solves 
$ (-\Delta)^s \tilde \phi = g -\lambda \phi$ in $\Omega$ with $\tilde \phi$ constant in $\RR\setminus \Omega$.
Standard regularity theory for the fractional Dirichlet boundary value problem (Proposition \ref{prop:reg}) implies that $\tilde \phi \in C^s(\RR)$ and so $\phi \in C^s(\overline\Omega)$.
In turns, this implies that 
$(-\Delta)_N^s\phi = (-\Delta)^s \tilde \phi  \in C^s(\overline\Omega)$. It is now easy to see that $\psi \in X^\beta_0$.

Using the fact that $\pa_t \psi(t,x) -\kappa (-\Delta)_N^s [\psi](t,x) = -e^{-\lambda t} g(x)$, it follows from \eqref{eq:weakFracNeumannuni} that
$$
\int_\Omega  \int_0^\infty \bar \rho(t,x)\, e^{-\lambda t}\, dt g(x)\, dx =0.
$$
Since this holds for all $g\in \mathcal D(\Omega)$, we deduce
$$ \int_0^\infty \bar \rho(t,x)\, e^{-\lambda t}\, dt  = 0 \quad \mbox{ in $\Omega$, for all } \lambda>0$$
and taking the inverse Laplace transform implies that $\bar \rho(t,x) =0$ in $\RR_+\times\Omega$.

\medskip

\end{proof}

\section{Optimal regularity for the elliptic problem: Proof of Proposition \ref{prop:reg}}

In this section, we are interested in the optimal regularity of the solutions to 
\begin{equation} \label{eq:elliptic} 
\left\{ \begin{aligned} 
& u+(-\Delta)_N^s u = g &\quad \mbox{ in } \Omega \\
& D_N^{2s-1} [u](x) = 0 & \quad \mbox{ on } \dO
\end{aligned} \right. 
\end{equation} 
with $\Omega = (0,1)$ and $g\in C^\infty(\Omega)$. As stated before, the extension $\tilde{u}$ solves $\tilde u+\Dels \tilde{u}= g$ and is constant outside $\Omega$ so the regularity theory for the fractional Dirichlet boundary value problems ensures that $u \in C^s(\bO)$. We will show that this regularity is optimal by constructing a solution of \eqref{eq:elliptic} which behaves like $\mbox{dist}(x,\dO)^s$ close to the boundary.

First, we recall that $v:x\mapsto \kappa_s x_+^s (1-x)_+^s$ is an explicit solution to 
\begin{equation*}
\left\{ \begin{aligned} 
& (-\Delta)^s v = 1 \quad & \mbox{ in } \Omega\\
& v = 0 & \mbox{ in } \RR\setminus\Omega 
\end{aligned} \right. 
\end{equation*} 
with $\Omega=(0,1)$ and the proper choice of constant $\kappa_s>0$, see e.g. \cite{RS}. 
Of course, this function does not satisfies $D^{2s-1}_N[v]=0$ on $\pa\Omega$.

We thus consider two smooth functions $\psi_0$ and $\psi_1$ with compact support in $(1/4,3/4)$ and such that for $i\in\{1,2\}$ and $x\in\dO$: $D^{2s-1}_N [\psi_i] (x) \neq 0$ and 
\begin{align*}
D^{2s-1}_N[ \psi_0] (0) D^{2s-1}_N[\psi_1](1) - D^{2s-1}_N[\psi_0](1)D^{2s-1}_N[ \psi_1](0) \neq 0. 
\end{align*}
We then define the function $u$ as 
\begin{align} \label{eq:counterexple}
u(x) = \kappa_s x_+^s (1-x)_+^s + \lambda_0 \psi_0(x) + \lambda_1 \psi_1(x) 
\end{align}
where $\kappa$ is a positive constant and the $\lambda_i$ are defined, similarly to Proposition \ref{prop:psieps}, as
\begin{equation*} 
	\left\{ \begin{aligned}
		\lambda_0 = \frac{D^{2s-1}_N[\kappa_s x_+^s (1-x)_+^s](0) D^{2s-1}_N[\psi_1](1) + D^{2s-1}_N[\kappa_s x_+^s (1-x)_+^s](1)D^{2s-1}_N[\psi_1](0)}{D^{2s-1}_N [\psi_0] (0) D^{2s-1}_N[\psi_1](1) - D^{2s-1}_N[ \psi_0](1) D^{2s-1}_N[\psi_1](0)} ,\\
		\lambda_1 = \frac{D^{2s-1}_N[\kappa_s x_+^s (1-x)_+^s](1) D^{2s-1}_N[\psi_0](0) +D^{2s-1}_N[\kappa_s x_+^s (1-x)_+^s](0)D^{2s-1}_N[\psi_0](1)}{D^{2s-1}_N[\psi_0] (0) D^{2s-1}_N[\psi_1](1) - D^{2s-1}_N [\psi_0](1) D^{2s-1}_N [\psi_1](0)}. 
	\end{aligned} \right.
\end{equation*}
Note that $x\to x_+^s (1-x)_+^s\in C^s(\bO)$ hence $D^{2s-1}_N [\kappa_s x_+^s (1-x)_+^s]$ is continuous on $\bO$ and the boundary values exist. This choice of $\lambda_i$ and the linearity of $D^{2s-1}_N$ implies naturally 
\begin{align*}
D^{2s-1}_N [u](x) = 0, \qquad \forall x\in\dO.
\end{align*}

Finally,  $u$ is a solution of \eqref{eq:elliptic} with right hand side
\begin{align*}
g(x) =  u+ 1 + \lambda_0 (-\Delta)^s_N \psi_0 + \lambda_1 (-\Delta)^s_N \psi_1 
\end{align*}
and by assumption on the support of $\psi_i$ we have $ (-\Delta)^s_N \psi_i=(-\Delta)^s \widetilde{\psi}_i  \in C^\infty(\Omega)$ and in particular $g\in L^\infty(\Omega)$. 

We have thus built a solution $u$ of \eqref{eq:elliptic} with $g\in L^\infty(\Omega)$, which behaves like $\mbox{dist}(x,\dO)^s$ when $x\to 0$ and $x\to 1$, which completes the proof.



\bibliographystyle{siam}
\bibliography{bibliography.bib}

\end{document}